\documentclass[preprint, 10pt, english]{elsarticle}
\usepackage{amsthm}
\usepackage{amsmath}
\usepackage{latexsym, amssymb}
\usepackage{txfonts}
\usepackage{mathtools}
\usepackage{color}
\usepackage{babel}
\usepackage[all]{xy}
\usepackage[capitalise]{cleveref}

\newtheorem{thm}{Theorem}[section] 

\newtheorem{cor}[thm]{Corollary}

\newtheorem{const}[thm]{Construction}

\newtheorem{lem}[thm]{Lemma}
\newtheorem{prop}[thm]{Proposition}

\theoremstyle{definition}
\newtheorem{rem}[thm]{Remark}
\newtheorem{exmpl}[thm]{Example}

\newcommand\operA[2]{{\if!#2!\operatorname{#1}\else{\operatorname{#1}_{#2}^{\phantom{I}}}\fi}} 

\newcommand\charac[1]{\mathrm{char}\left(#1\right)}

\newcommand{\gr}{\mathsf{gr}}

\newcommand{\Trace}[1][]{\if!#1!\operatorname{Tr}\else{\operatorname{Tr}_{#1}^{\phantom{I}}}\fi} 

\long\def\forget#1\forgotten{{}} %

\def\({\left(}
\def\){\right)}

\newcommand\LAY[3][]{{\begin{array}{c}\mbox{#2} \if#1!{}\else{+}\fi \\ \mbox{#3}\end{array}}}

\makeatletter

\def\ps@pprintTitle{%
 \let\@oddhead\@empty
 \let\@evenhead\@empty
 \def\@oddfoot{}%
 \let\@evenfoot\@oddfoot}

\newcommand{\bigperp}{%
  \mathop{\mathpalette\bigp@rp\relax}%
  \displaylimits
}

\newcommand{\bigp@rp}[2]{%
  \vcenter{
    \m@th\hbox{\scalebox{\ifx#1\displaystyle2.1\else1.5\fi}{$#1\perp$}}
  }%
}
\makeatother

\renewcommand{\geq}{\geqslant}
\renewcommand{\leq}{\leqslant}

\newif\iffurther
\furtherfalse

\journal{??}

\begin{document}
\begin{frontmatter}

\title{The $m$-Invariant of Fields of Characteristic 2}

\author{Adam Chapman}
\ead{adam1chapman@yahoo.com}
\address{School of Computer Science, Academic College of Tel-Aviv-Yaffo, Rabenu Yeruham St., P.O.B 8401 Yaffo, 6818211, Israel}
\author{Kelly McKinnie}
\ead{kelly.mckinnie@mso.umt.edu}
\address{Department of Mathematics, University of Montana, Missoula, MT 59812, USA}

\begin{abstract}
We study the minimal dimension of a nonsingular anisotropic universal quadratic form over a field $F$ of $\operatorname{char}(F)=2$, which we call the ``$m$-invariant" of $F$. We show that every even number is the $m$-invariant of some field of characteristic 2, and consider cases when $u(F)$ and $m(F)$ are equal and when they are not.
\end{abstract}

\begin{keyword}
Quadratic Forms; Universal Quadratic Forms; Fields of Positive Characteristic
\MSC[2020] 11E81 (primary); 11E04 (secondary)
\end{keyword}
\end{frontmatter}

\section{Introduction}

The $m$-invariant of fields of characteristic not 2 was introduced in \cite{GesquiereVanGeel:1992} and received serious attention only much later in the papers \cite{Cassady:2026}, \cite{Cassady:2026a} and \cite{Cassady:2026b}. In particular, in \cite{Cassady:2026a} Cassady shows that for each integer $n \ne 3,5$, there exists a field of characteristic not 2 with $m$-invariant equal to $n$. Here we want to carry out a similar study of this invariant for fields of characteristic 2. 

Recall that the definition of a related invariant, the $u$-invariant $u(F)$, in characteristic 2 is the maximal dimension of an anisotropic \textbf{nonsingular} quadratic form, see for example \cite{MammoneTignolWadsworth:1991}. The nonsingularity requirement is redundant in characteristic not 2, but in characteristic 2 it is not.
Similarly, we define the $m$-invariant of $F$, $m(F)$, in characteristic 2 to be the minimal dimension of an anisotropic nonsingular {\it universal} quadratic form over $F$. Following \cite{GesquiereVanGeel:1992}, we set $m(F)=\infty$ if there are no universal anisotropic quadratic forms. Note that when $u(F)=0$ (the case where no anisotropic nonsingular forms exist), we have $m(F)=\infty>u(F)$, but when $u(F)\geq 2$, we have $m(F)\leq u(F)$ (see Lemma \ref{NuLemma}).

{\it Main Results.} In Section 3 we show that $m(F)$ is bounded from above by $2^{\nu(F)}$, which provides a family of fields $F$ with $m(F)<u(F)$; In Section 4 we show that every even number is the $m$-invariant of some  field of characteristic 2. Since a nonsingular quadratic form over a field of characteristic 2 has even dimension, this covers all possible values of $m(F)$. In Section 5 we give an infinite family of fields of characteristic 2 for which $m(F)=u(F)=2^n$. 
\section{Background}

A quadratic form $q : V \rightarrow F$ over a field $F$ admits an underlying symmetric bilinear form $B_q : V \times V \rightarrow F$ given by $B(v,w)=q(v+w)-q(v)-q(w)$. This correspondence is 1-to-1 when $\operatorname{char}(F)\neq 2$, but when $\operatorname{char}(F)=2$, it is not. We say that $q$ is nonsingular (or singular, respectively) if the matrix $M_q$ satisfying $vM_q w^t=B_q(v,w)$ is nonsingular (singular, resp.).

From now on, assume $\operatorname{char}(F)=2$. When $q$ is nonsingular, it decomposes as $q=b_1[1,a_1]\perp \dots \perp b_n[1,a_n]$ where $\perp$ is the direct sum and $[1,a]$ stands for the quadratic form $x^2+xy+ay^2$.

A nondegenerate quadratic form is either a nonsingular form (in the even dimensional case) or $q \perp \langle a\rangle$ (in the odd dimensional case). The Witt index $i_0(q)$ of a quadratic form is the maximal dimension of a totally isotropic subspace of $q$. When $q$ is nonsingular, it decomposes (uniquely) as $q_{an}\perp i_0(q)\times \varmathbb{H}$ where $\varmathbb{H}=[1,0]$ is the hyperbolic plane (i.e., the unique nonsingular isotropic 2-dimensional form). This is slightly more complicated when $q$ is odd-dimensional and nondegenerate, for example $[1,a]$ may be anisotropic, and $q=[1,a] \perp \langle a \rangle \simeq \varmathbb{H}\perp \langle a \rangle$. In particular, Witt cancellation holds true only for nonsingular forms and not nondegenerate forms. Nevertheless, there are a few notable results, such as the results from \cite[Section 74]{EKM} that hold true for nondegenerate forms, that we use later in this paper. The first Witt index $i_1(q)$ is defined to be $i_0(q_L)$ where $L=F(q)$. It is clearly a positive integer, but can sometimes be greater than 1, e.g., when $q$ is a Pfister form.

Witt equivalence is given by $\varphi \sim_{Witt} \varphi \perp n\times \varmathbb{H}$ for any $n$, and the Witt group $W_q F$ or $I_q F$ is defined to be the group of Witt classes of nonsingular quadratic forms with $\perp$ as the binary operation. The Arf invariant of $q=b_1[1,a_1]\perp \dots \perp b_n[1,a_n]$ is given by the class of $a_1+\dots+a_n$ in $F/\wp(F)$ where $\wp(F)=\{\lambda^2+\lambda : \lambda\in F\}$. The subgroup $I_q^2 F$ is the kernel of the Arf invariant. The Clifford invariant maps $\varphi \in I_q^2 F$ to its Clifford algebra $C\ell(\varphi) \in Br(F)$. The kernel of this map is $I_q^3 F$. The formula of the Clifford invariant is given by
$$b_1[1,a_1]\perp \dots \perp b_{n-1}[1,a_{n-1}]\perp [1,a_1+\dots+a_{n-1}]\mapsto [a_1,b_1)_{2,F} \otimes \dots \otimes [a_{n-1},b_{n-1})_{2,F},$$
where $[a,b)_{2,F}$ is the quaternion algebra generated by $i,j$ satisfying $i^2+i=a$, $j^2=b$ and $jij^{-1}=i+1$.

\section{Quadratic $n$-Fold Pfister Forms}

More generally, a quadratic $n$-fold Pfister form is of the form
$$ \langle \! \langle b_1,\dots,b_{n-1},a]\!]=\perp_{i_1,\dots,i_{n-1} \in \{0,1\}} b_1^{i_1} \dots b_{n-1}^{i_{n-1}} [1,a],$$
where $b_1,\dots,b_{n-1}\in F^\times$ and $a\in F$. Such forms can be either anisotropic or hyperbolic, and their scalar multiples generate $I_q^n F$. The quotient group $I_q^n F/I_q^{n+1} F$ is known to be isomorphic to $H_2^n(F)$ defined via differential forms (see \cite{Kato:1982}).

One way to produce interesting examples of fields $F$ with $u(F)\neq m(F)$ is by considering another field invariant, $\nu(F)$, which is the minimal $n$ for which $H_2^{n+1}(F)=0$. Equivalently, $\nu(F)$ is the minimal $n$ so that $I_q^{n+1} F=0$ (see \cite{AravireBaeza:1989} and \cite{ArasonAravireBaeza:2007}).

\begin{lem}\label{NuLemma}
    When $\operatorname{char}(F)=2$, $m(F)\leq 2^{\nu(F)}$.
\end{lem}

\begin{proof}
    Set $n=\nu(F)$. Then there exists an anisotropic quadratic $n$-fold Pfister form $\varphi$ over $F$. This form is universal, because for any $c\in F^\times$, $\varphi\perp \langle c \rangle$ is isotropic, for it is a Pfister neighbor of the $(n+1)$-fold Pfister form $\langle \! \langle c \rangle \! \rangle \otimes \varphi$, which must be hyperbolic for $\nu(F)<n+1$, whereas $\varphi$ is anisotropic, which means $c$ is represented by $\varphi$.
\end{proof}

\begin{rem}
    A similar $2^n$-type upper bound on values of the $m$-invariant holds in the case of fields of characteristic not 2. In particular, as mentioned in \cite[Proposition 2.3]{Cassady:2026b}, when $\charac{F} \ne 2$ and $n$ is the largest integer so that $2^n\leq u(F)$ one has $m(F)\leq 2^n$. This implies that if $m(F) = u(F)$ then they are both a power of 2. The characteristic 2 $m(F)=u(F)$ examples in Corollary \ref{u_equals_m} follow this pattern, satisfying $m(F)=u(F) = 2^n$.
\end{rem}
\begin{exmpl}
    In \cite{MammoneTignolWadsworth:1991} (and \cite[Theorem 38.4]{EKM}) fields $F$ of characteristic 2 and $I_q^3 F=0$ and $u(F)=2n$ for any $n\in \mathbb{N}$ were constructed. Suppose $n\geq 3$.
    For these fields, the $u$-invariant is also the maximal dimension of an anisotropic quadratic form in $I_q^2 F$ (see \cite[Lemma 4.3]{Chapman:2017}), and so $I_q^2 F\neq 0$, which means $\nu(F)=2$. Therefore, we have $m(F)\leq 2^{\nu(F)}=4 < 6\leq u(F)$.
\end{exmpl}

\section{Possible Values of the $m$-invariant}

To build fields whose $m$-invariant is any even number, we adapt the techniques of \cite[Prop 3.2]{Cassady:2026a} to the characteristic 2 case using characteristic free results in \cite{EKM}. The argument in \cite[Prop 3.2]{Cassady:2026a} is itself based on a similar construction found in \cite[Prop 4.3]{Hoffmann:1994}. The basic tools on the first Witt index of subforms are taken from \cite[Section 74]{EKM}:

\begin{lem}[{\cite[Lemma 74.1 (2) and Proposition 74.2]{EKM}}]\label{EKMlem1}
Let $\varphi$ be an anisotropic nondegenerate quadratic form of dimension at least 2 over a field $F$, and let $\psi$ be a nondegenerate subform of $\varphi$ of codimension $r$.
\begin{enumerate}
    \item If $E/F$ is a field extension, then $i_0(\psi_E)\geq i_0(\varphi_E)-r$.
    \item If $r \geq i_1(\varphi)$, then the form $\psi_{F(\varphi)}$ is anisotropic.
\end{enumerate}
\end{lem}

\begin{lem}\label{remain}
    Let $q$ be an anisotropic nonsingular form of dimension at least $2$ over a field $L$ of characteristic 2, and let $a\in L^\times$ be such that $q_M$ is anisotropic for $M=L(q\perp \langle a \rangle)$. Then all anisotropic nondegenerate quadratic  forms $p$ over $L$ with $1\leq \dim p \leq \dim q$ remain anisotropic over $M$.
\end{lem}

\begin{proof}
Set $m=\dim q$ and $\varphi=q\perp \langle a \rangle$.
Since $q$ is a codimension one subform of $\varphi$ and $q_{L(\varphi)}$ is anisotropic, by Lemma \ref{EKMlem1}(1)
$$0=i_0(q_{L(\varphi)})\geq i_1(\varphi)-1,$$
and so $i_1(\varphi)=1$.

Suppose $p$ is an anisotropic nondegenerate quadratic form of dimension $d\leq m$ that becomes isotropic over $L(\varphi)$. Let $X$ and $Y$ be the projective quadrics of $\varphi$ and $p$. Recall that $\dim_{\mathrm{Izh}}X=\dim X-i_1(\varphi)+1$. Then \cite[Theorem 76.5]{EKM}  gives
$$\dim_{\mathrm{Izh}}X\leq \dim_{\mathrm{Izh}}Y.$$
But
$$\dim_{\mathrm{Izh}}X
  =(\dim\varphi-2)-i_1(\varphi)+1
  =m-1,$$
whereas
$$
\dim_{\mathrm{Izh}}Y
  =d-i_1(p)-1
  \leq d-2
  \leq m-2,$$
a contradiction.
\end{proof}

\begin{const}\label{cc}
Let $k$ be a field of characteristic 2 over which there exists an anisotropic nonsingular form $\psi$ of $\dim \psi \geq 2$. Define $E_i$ for $i\geq 0$ inductively as follows:
$\tilde{E}_i=E_i(X_i)$ for a transcendental element $X_i$, $E_0=k$ and for $i\geq 1$, $E_i=\tilde{E}_{i-1}(\{ \psi \perp \langle a \rangle : a\in E^\times_{i-1}\})$, the free compositum of the corresponding function fields.
Set $F=\bigcup_{i=0}^\infty E_i$.
\end{const}

\begin{prop}\label{mm}
    Let $k,\psi,E_i,F$ be as in Construction \ref{cc}. If $\psi$ remains anisotropic over $F$, then $m(F)=\dim \psi$.
\end{prop}

\begin{proof}
    Consider an arbitrary element $a\in F^\times$. Then $a\in E^\times_i$ for some $i\geq 0$. Hence $\psi \perp \langle a \rangle$ is isotropic over $E_{i+1}$ and therefore over $F$. Therefore, the anisotropic form $\psi$ represents $a$ over $F$. Consequently, $\psi$ is universal, and thus $m(F)\leq \dim \psi$.

    Now, let $q$ be an anisotropic nonsingular quadratic form over $F$ of dimension strictly smaller than $\dim \psi$. This form descends to $E_i$ for some $i\geq 0$. It is easy to see that $\pi=q_{\tilde{E}_i}\perp \langle X_i \rangle$ is anisotropic. In order to argue that $q$ is not universal, it suffices to show that $\pi$ remains anisotropic under restriction to $E_j$ for any $j>i$.

    Since $\pi$ is anisotropic over $\tilde{E}_i$, if $\pi_{E_{i+1}}$ were isotropic, then $\pi_T$ would be isotropic for the function field $T$ of $\psi \perp \langle a_0 \rangle$ for some $a_0\in E_i^\times$. However, $\psi$ is anisotropic over $E_{i+1}$, and so, it is anisotropic over the function field of $\psi \perp \langle a \rangle$ for any $a\in E^\times_i$, and in particular for $a=a_0$. Therefore, since $\dim \pi =\dim q+1\leq\dim \psi$, $\pi$ remains anisotropic over $E_{i+1}$ by Lemma \ref{remain}. This implies that $\pi$ is anisotropic over $\tilde{E}_{i+1}$. Repeating this argument for any finite number of steps shows that $\pi_{E_j}$ is anisotropic for any $j>i$.
\end{proof}

There is yet another tool we need, this time from \cite[Section 30]{EKM}:

\begin{lem}[{\cite[Corollary 30.9]{EKM}}]\label{ELMlem2}
If $D$ is a division algebra of degree smaller than $2^n$ over $F$ and $\varphi$ a nondegenerate anisotropic form of dimension at least $2n+1$, then $D_{F(\varphi)}$ is a division algebra.
\end{lem}

\begin{thm}\label{evens}
    For each positive $n\geq 2$ there exists a field $F$ of $\operatorname{char}(F)=2$ and $m(F)=2n$.
\end{thm}

\begin{proof}
    Let $k$ be a field of characteristic 2 admitting a division algebra $D$ that decomposes as a tensor product of $n-1$ quaternion algebras 
    $$D=[a_1,b_1)_{2,k} \otimes \dots \otimes [a_{n-1},b_{n-1})_{2,k}.$$
    See \cite{Chapman:2020} for examples of such fields and algebras. Then $D$ has index $2^{n-1}$.
    Taking $\psi=b_1[1,a_1]\perp \dots \perp b_{n-1}[1,a_{n-1}]\perp [1,a_1+\dots+a_{n-1}]$, we have $C\ell(\psi)\cong M_2(D)\sim_{Br} D$. In particular, $\psi$ is an anisotropic form of dimension $2n$.
    Now follow Construction \ref{cc} starting with these $k$ and $\psi$ to create $E_i$ for $i\geq 0$ and $F$.

    Given a division algebra $B$ over a field $K$, the restriction $B_{K(x)}$ to the field of rational functions in one variable $x$ over $K$ is also a division algebra.
    Moreover, if the index of $B$ is $2^{n-1}$, then $B_{K(\varphi)}$ is a division algebra for any nondegenerate anisotropic form $\varphi$ of dimension $\geq 2n+1$ by Lemma \ref{ELMlem2}.
    Since these are the two kinds of steps that are taken in moving from $E_i$ to $\tilde E_i$ and from $\tilde E_i$ to $E_{i+1}$, $D_F$ is a division algebra, and so $\psi_F$ is anisotropic. It follows from Proposition \ref{mm} that $m(F)=2n$.
\end{proof}

\begin{rem}
    It is not difficult to produce an example of a field $F$ of $m(F)=2$. Take for example $F=\mathbb{F}_2$. This covers the case of $n=1$ not addressed in the previous theorem.
\end{rem}

\section{Equality between $m(F)$ and $u(F)$}

Let $(F,v)$ be a field of arbitrary characteristic with a maximally
complete valuation $v\colon F\to\Gamma\cup\{\infty\}$. Write $k$ for
the residue field of $F$, which may not have the same characteristic
as $F$, and $\Gamma_F$ for $v(F^\times)$. Recall that for any base field $k$ the iterated Laurent Series field $F=k(\!(x_1)\!)(\!(x_2)\!)\cdots(\!(x_n)\!)$ with value group $\mathbb Z^n$ is maximally complete (e.g., \cite[Example 3.11]{TignolWadsworth:2015}). 

Let also $(V,q)$ be a finite-dimensional $K$-vector space with a
quadratic form $q\colon V\to F$, which may be degenerate, and write
$b$ for the polar form of $q$. Let $\Gamma_q\subseteq \Gamma_F$ be the subset $\Gamma_q =\{v(q(x))\mid x\in V\setminus\{0\}\}$. Note that $\Gamma_q$ may not be a group, but because $v(q(\lambda x))) = 2v(\lambda)+v(q(x))$, it is a set of cosets of the subgroup $2\Gamma_F$. Denote the number of cosets by $\left\vert \Gamma_q/2\Gamma_F\right \vert$. In Proposition \ref{Tignol} we show that the number of cosets is equal to the dimension of $V$, hence giving a possible obstruction to a form being universal. This is used in Corollary \ref{u_equals_m}.

\begin{prop}\label{Tignol}
  If $q$ is anisotropic and $k$ is algebraically closed, then
  \[
    \left\vert \Gamma_q/2\Gamma_F\right\vert = \dim_FV.
  \]
\end{prop}

\begin{proof}
  Define a map $\alpha\colon V\to \frac{1}{2}\Gamma_q\cup\{\infty\}$ by
  \[
    \alpha(x)= \textstyle{\frac12}v(q(x)) \qquad\text{for $x\in V$.}
  \]
  Since $q$ is anisotropic and $v$ is Henselian, $\alpha$ is a $v$-value
  function on $V$ such that
  \[
    \alpha(x)+\alpha(y)\leq v(b(x,y)) \quad\text{and}\quad
    2\alpha(x)\leq v(q(x)) \quad\text{for all $x$, $y\in V$,}
  \]
  see \cite[Prop.~12]{ElomaryTignol:2011}. We may therefore define an induced graded
  quadratic form $\widehat q\colon \gr_\alpha(V) \to \gr(F)$ with polar
  form $\widehat b\colon \gr_\alpha(V)\times\gr_\alpha(V)\to \gr(F)$ such
  that
  \begin{equation}
    \label{eq:1}
    \widehat q(\widetilde x) = q(x)+F_{>2\alpha(x)} \quad\text{and}\quad
    \widehat b(\widetilde x,\widetilde y) = b(x,y)+F_{>\alpha(x)+\alpha(y)}
    \quad\text{for all $x$, $y\in V\setminus\{0\}$},
  \end{equation}
  see \cite[\S3]{ElomaryTignol:2011}. With $\Gamma_V:=\alpha(V\setminus\{0\})$ we
  have
  \[
    \Gamma_V=\{\textstyle{\frac12}v(q(x))\mid x\in V\setminus\{0\}\},
  \]
  and we claim that $\left\vert\Gamma_V/\Gamma_F\right\vert=\dim_KV$. Note that $\left\vert\Gamma_V/\Gamma_F\right\vert = \left\vert \Gamma_q/2\Gamma_F\right\vert$.

  Picking a representative $\gamma_i\in\Gamma_V$ in each coset
  $i\in\Gamma_V/\Gamma_F$ and letting $V_{\gamma_i}$ denote the
  homogeneous component of degree $\gamma_i$ in $\gr_\alpha(V)$, we
  obtain from~\cite[Prop.~2.5]{TignolWadsworth:2015}
  \begin{equation}
    \label{eq:2}
    \dim_{\gr(F)}\gr_\alpha(V) = \sum_{i\in \Gamma_V/\Gamma_F}
    \dim_kV_{\gamma_i}.
  \end{equation}
  Now, as defined in~\eqref{eq:1}, $\widehat q$ yields a quadratic
  ``form'' $q_{\gamma_i}\colon V_{\gamma_i}\to F_{2\gamma_i}$, where
  the homogeneous component $F_{2\gamma_i}$ of degree~$2\gamma_i$ in
  $\gr(F)$ is a $1$-dimensional $k$-vector space. From the definition
  of $\alpha$, it is clear that $q_{\gamma_i}$ is anisotropic, hence
  $\dim_kV_{\gamma_i}=1$ since $k$ is algebraically closed. Therefore,
  \eqref{eq:2} yields
  \[
    \dim_{\gr(F)}\gr_\alpha(V)=\left\vert\Gamma_V/\Gamma_F\right\vert.
  \]
  Since $F$ is maximally complete, the value function $\alpha$ is a
  $v$-norm, hence 
  $\dim_FV=\dim_{\gr(F)}\gr(V)$ and the proposition follows.
\end{proof}

\begin{cor} \label{u_equals_m}
Let $k$ be algebraically closed of characteristic $2$ and
$F=k(\!(x_1)\!)\cdots(\!(x_n)\!)$.
Then $u(F)=m(F)=2^n$.
\end{cor}

\begin{proof}
    Recall that $F$ is a $C_n$-field (\cite[Chapter 6]{EnglerPrestel:2005}), which means that every homogeneous polynomial form of degree $d$ in more than $d^n$ variables is isotropic. In particular, for $d=2$ it means that $u(F)\leq 2^n$. On the other hand, $\varphi=\langle \! \langle x_1^{-1},\dots,x_n^{-1} \rangle \! \rangle$ is an anisotropic quadratic $n$-fold Pfister form over $F$, which means $u(F)=2^n$.
    The form $\varphi$ is also universal, because of the $C_n$-property.

If $\psi$ is a quadratic form of dimension $< 2^n$ over $F$ then by Proposition \ref{Tignol}, $\left\vert \Gamma_\psi/2\Gamma_F\right\vert <2^n$. Since $\left \vert \Gamma_F/2\Gamma_F\right\vert = 2^n$, $\Gamma_\psi \subsetneq \Gamma_F$, showing that $\psi$ is not universal. Therefore, $m(F)=2^n$.
    
\end{proof}

\section*{Acknowledgments}

We are indebted to Jean-Pierre Tignol for providing the proof of Proposition \ref{Tignol}.

\bibliographystyle{abbrv}
\bibliography{bibfile}

\end{document}